\documentclass[11pt]{article}
\usepackage[margin=1in]{geometry} 
\usepackage{amssymb,amsfonts,amsmath,bbm,mathrsfs,stmaryrd}
\usepackage{xcolor}
\usepackage{url}

\usepackage[colorlinks,
             linkcolor=black!75!red,
             citecolor=blue,
             pdftitle={Free unitary groups are (almost) simple},
             pdfauthor={Alexandru Chirvasitu},
             pdfsubject={rings and algebras},
             pdfkeywords={qa.quantum algebra, rt.representation theory, co.combinatorics},
             pdfproducer={pdfLaTeX},
             pdfpagemode=None,
             bookmarksopen=true
             bookmarksnumbered=true]{hyperref}

\usepackage{tikz}
\usetikzlibrary{arrows,decorations.pathreplacing,decorations.markings,shapes.geometric,through,fit,shapes.symbols,positioning}

\usepackage[amsmath,thmmarks,hyperref]{ntheorem}
\usepackage{cleveref}

\crefname{section}{Section}{Sections}
\crefformat{section}{#2Section~#1#3} 
\Crefformat{section}{#2Section~#1#3} 

\crefname{subsection}{\S}{\S\S}
\crefformat{subsection}{#2\S#1#3} 
\Crefformat{subsection}{#2\S#1#3} 

\theoremstyle{change}

\newtheorem{lemma}{Lemma}[subsection]
\newtheorem{proposition}[lemma]{Proposition}
\newtheorem{corollary}[lemma]{Corollary}
\newtheorem{theorem}[lemma]{Theorem}
\newtheorem{question}[lemma]{Question}

\theoremstyle{nonumberplain}
\newtheorem{theoremN}{Theorem}

\theoremstyle{change}
\theorembodyfont{\upshape}
\theoremsymbol{\ensuremath{\blacklozenge}}

\newtheorem{de}[lemma]{Definition}
\newtheorem{ex}[lemma]{Example}
\newtheorem{re}[lemma]{Remark}

\crefname{definition}{definition}{definitions}
\crefformat{definition}{#2definition~#1#3} 
\Crefformat{definition}{#2Definition~#1#3} 

\crefname{lemma}{lemma}{lemmas}
\crefformat{lemma}{#2lemma~#1#3} 
\Crefformat{lemma}{#2Lemma~#1#3} 

\crefname{proposition}{proposition}{propositions}
\crefformat{proposition}{#2proposition~#1#3} 
\Crefformat{proposition}{#2Proposition~#1#3} 

\crefname{example}{example}{examples}
\crefformat{example}{#2example~#1#3} 
\Crefformat{example}{#2Example~#1#3} 

\crefname{remark}{remark}{remarks}
\crefformat{remark}{#2remark~#1#3} 
\Crefformat{remark}{#2Remark~#1#3} 

\crefname{corollary}{corollary}{corollaries}
\crefformat{corollary}{#2corollary~#1#3} 
\Crefformat{corollary}{#2Corollary~#1#3} 

\crefname{theorem}{theorem}{theorems}
\crefformat{theorem}{#2theorem~#1#3} 
\Crefformat{theorem}{#2Theorem~#1#3} 

\crefname{equation}{}{}
\crefformat{equation}{(#2#1#3)} 
\Crefformat{equation}{(#2#1#3)}

\theoremstyle{nonumberplain}
\theoremsymbol{\ensuremath{\blacksquare}}

\newtheorem{proof}{Proof}
\newtheorem{proof of main}{Proof of \Cref{th.main}}

\DeclareMathOperator{\id}{id}

\newcommand\bN{{\mathbb N}}
\newcommand\bS{{\mathbb S}}
\newcommand\bZ{{\mathbb Z}}

\DeclareMathOperator{\cqg}{\mathrm{CQG}}

\DeclareMathOperator{\lker}{\cat{LKer}}
\DeclareMathOperator{\rker}{\cat{RKer}}

\newcommand{\define}[1]{{\em #1}}

\newcommand{\cat}[1]{\textsc{#1}}

\newcommand{\qedhere}{\mbox{}\hfill\ensuremath{\blacksquare}}

\title{Free unitary groups are (almost) simple}
\author{Alexandru Chirvasitu\footnote{University of California at Berkeley, \url{chirvasitua@math.berkeley.edu}}}
\begin{document}
\maketitle

\begin{abstract} 
We show that the quotients of Wang and Van Daele's universal quantum groups by their centers are simple in the sense that they have no normal quantum subgroups, thus providing the first examples of simple compact quantum groups with non-commutative fusion rings.  
\end{abstract}


\tableofcontents

%%%%%%%%%%%%%%%%%%%%%%%%%%%%%%%%%%%%%%%%%%%%%%%%%%%%%%%%%%%%%%%%%%%%%%%%%%%%%%%%%%%%%%%%%%%%%%%%%%%%%%%%%%%%%%%%%%
%%%%%%%%%%%%%%%%%%%%%%%%%%%%%%%%%%%%%%%%%%%%%%%%%%%%%%%%%%%%%%%%%%%%%%%%%%%%%%%%%%%%%%%%%%%%%%%%%%%%%%%%%%%%%%%%%%

\section*{Introduction}

A short note such as this one can hardly do justice to the richness of the subject of quantum groups, so we will simply refer the reader to the papers mentioned below, as well as to the vast literature cited by those papers, for a comprehensive view of the history and intricacies of the subject. Let us just mention that the quantum groups in this paper are of function algebra type, rather than of the quantized universal enveloping algebra kind introduced in \cite{MR797001,MR934283}. The former are non-commutative analogues of algebras of continuous functions on a (here compact) group, and were formalized in essentially their present form in \cite{Woronowicz1987}. To be more specific, we only work with the algebraic counterparts of the objects described by Woronowicz. These are the so-called CQG algebras of \cite{MR1310296}, and mimic the algebras of \define{representative} functions on a compact group, minus the commutativity. They are complex Hopf $*$-algebras satisfying an additional condition whih ensures, among other things, the semisimplicity of their categories of comodules. 

It makes sense, in view of the importance and ubiquity of \define{simple} compact Lie groups, to study analogous notions in the quantum setting. This program was initiated in \cite{MR2504527}, where simple compact quantum groups are defined (along with the notion of normal compact quantum subgroup) and examples are provided. It is observed there that these examples are all \define{almost classical}, in the sense that their so-called fusion algebras (meaning Grothendieck rings of their categories of finite-dimensional comodules) are isomorphic to fusion algebras of compact Lie groups. 

One source of compact quantum groups is provided by ordinary compact Lie groups deformed in some sense (\cite{MR1116413,MR1216204} and references therein), but more exotic examples can be obtained as quantum automorphism groups of various structures, such as, say, bilinear forms \cite{MR1068703}, finite-dimensional $C^*$-algebras endowed with a trace \cite{MR1637425,MR1709109}, finite graphs \cite{MR1937403}, finite metric spaces \cite{MR2174219}, etc. In the context of simplicity, it turns out \cite[$\S$5]{MR2504527} that deforming simple compact Lie groups produces, as expected, simple compact quantum groups. On the other hand, it is shown in the same paper ($\S$4) that quantum automorphism groups of traced finite-dimensional $C^*$-algebras are simple, while quantum automorphism groups of non-degenerate bilinear forms have a single non-trivial normal compact quantum subgroup of order $2$.  

The universal quantum groups $A_u(Q)$ (the free unitary groups of the title) parametrized by invertible positive matrices $Q$ were introduced in \cite{MR1382726}. Since they jointly play the same role as that of the family of unitary groups in the theory of compact Lie groups, they are arguably the first examples one should test compact quantum group notions or tentative results against. Wang notes in \cite[Proposition 4.5]{MR2504527} that $A_u(Q)$ always has a central one-dimensional torus (just like classical unitary groups), and hence cannot possibly be simple. The main result of this paper is that the quotient by this central circle group is nevertheless simple, again as in the classical case (see \Cref{se.prel} for an explanation of the terminology):

\begin{theoremN}
Let $Q\in GL_n$ be a positive invertible matrix. The quotient of the quantum group $A_u(Q)$ by its central circle subgroup $A_u(Q)\to C(\bS^1)$ is simple. 
\end{theoremN}

The outline of the paper: We recall the relevant terminology, conventions and results in \Cref{se.prel}, and prove the main results (one of which is stated above) in \Cref{se.main}.

%%%%%%%%%%%%%%%%%%%%%%%%%%%%%%%%%%%%%%%%%%%%%%%%%%%%%%%%%%%%%%%%%%%%%%%%%%%%%%%%%%%%%%%%%%%%%%%%%%%%%%%%%%%%%%%%%%

\subsection*{Acknowledgement} I am grateful to Shuzhou Wang for fruitful conversations on the contents of \cite{MR2504527}, and for his many suggestions on the improvement of this paper.

%%%%%%%%%%%%%%%%%%%%%%%%%%%%%%%%%%%%%%%%%%%%%%%%%%%%%%%%%%%%%%%%%%%%%%%%%%%%%%%%%%%%%%%%%%%%%%%%%%%%%%%%%%%%%%%%%%
%%%%%%%%%%%%%%%%%%%%%%%%%%%%%%%%%%%%%%%%%%%%%%%%%%%%%%%%%%%%%%%%%%%%%%%%%%%%%%%%%%%%%%%%%%%%%%%%%%%%%%%%%%%%%%%%%%

\section{Preliminaries}\label{se.prel}

We assume the basics of Hopf algebra and coalgebra theory, as covered in, say, \cite{MR1243637,MR1786197}. Comultiplications, counits, and antipodes are denoted as usual by $\Delta,\varepsilon$ and $S$ perhaps adorned with the name of the coalgebra or Hopf algebra as in $\Delta_H$, and we use Sweedler notation with suppressed summation symbol for comultiplication, as in $\displaystyle \Delta(x)=x_1\otimes x_2$. Comodules are always right and finite-dimensional, and everything in sight is complex. 

For CQG algebras, unitary and unitarizable comodules, and all other notions that go into formalizing compact quantum groups in a purely algebraic setting we refer to \cite{MR1310296} or \cite[$\S$11]{MR1492989}. Here, we only remind the reader that a CQG algebra is a Hopf $*$-algebra (Hopf algebra endowed with a conjugate linear, involutive algebra anti-automorphism `$*$' making both $\Delta$ and $\varepsilon$ morphisms of $*$-algebras) satisfying an additional condition which ensures that all finite-dimensional comodules admit an inner product invariant under the coaction in some sense (\cite[11.1.5]{MR1492989}). This accords with the point of view that representations of the compact quantum group are \define{co}modules over the Hopf algebra which is supposed to behave like the algebra of representative functions on the otherwise fictitious group. Denote by $\cqg$ the category of CQG algebras and Hopf $*$-algebra morphisms. 

We also take for granted the correspondence between comodules over a coalgebra $C$ and subcoalgebras of $C$, sending a $C$-comodule $V$ to the smallest subcoalgebra $C_V\subseteq C$ such that the structure map $V\mapsto V\otimes C$ factors through $V\otimes C_V$. This correspondence induces a bijection between isomorphism classes of simple comodules and simple subcoalgebras of $C$, and in a Hopf algebra $H$ it behaves well with respect to multiplication and the antipode: The coalgebra associated to the tensor product $V\otimes W$ is exactly the product $C_VC_W$ in $H$ (that is, the linear span of products of elements from $C_V$ and $C_W$), while the coalgebra corresponding to the dual $V^*$ is $S(C_V)$. 

The \define{fusion semiring} of a Hopf algebra is the Grothendieck semiring of its category of (finite-dimensional, right) comodules. Note that it admits a natural anti-endomorphism $\alpha\mapsto\alpha^*$, sending (the class of) a comodule to (the class of) its dual. When the Hopf algebra is cosemisimple (e.g. a CQG algebra), this anti-endomorphism is involutive. By a slight abuse of notation, if $x$ is the class of $V$ in the fusion semiring, we write $C_x$ for $C_V$.

%%%%%%%%%%%%%%%%%%%%%%%%%%%%%%%%%%%%%%%%%%%%%%%%%%%%%%%%%%%%%%%%%%%%%%%%%%%%%%%%%%%%%%%%%%%%%%%%%%%%%%%%%%%%%%%%%%

\subsection{Free unitary groups}\label{subse.free}

The main characters in this paper are the objects $A_u(Q)$ mentioned in the introduction. Here, $Q$ is a positive invertible $n\times n$ matrix, and by definition, $A=A_u(Q)$ is the $*$-algebra freely generated by $n^2$ elements $u_{ij}$, $i,j=\overline{1,n}$ subject to the conditions that both $u=(u_{ij})_{i,j}$ and $Q^{\frac 12}\overline uQ^{-\frac 12}$ be unitary $n\times n$ matrices in $A$ (where $\overline u=(u_{ij}^*)_{i,j}$); cf. \cite[11.3.1, Example 6]{MR1492989}, or \cite{MR1382726}, where these objects were first introduced in their $C^*$-algebraic versions. 

It turns out that $A$ can be made into a CQG algebra by demanding that $u_{ij}$ be the $n^2$ ``matrix units'' of an $n\times n$ matrix coalgebra: \[ \Delta(u_{ij}) = \sum_k u_{ik}\otimes u_{kj},\quad \varepsilon(u_{ij})=\delta_{ij}. \] The antipode is then defined by $S(u)=u^*=(u_{ji}^*)_{i,j}$, and one checks the CQG condition by observing that both the comodule $V$ with basis $e_i$, $i=\overline{1,n}$ whose comodule structure is defined by $e_j\mapsto \sum_i e_i\otimes u_{ij}$ and its dual are unitarizable; we refer to the cited sources for details on how this works. 

The reason why these are called `free unitary' is that every CQG algebra finitely generated as an algebra is a quotient of some $A_u(Q)$; keeping in mind the arrow reversal implicit in passing from groups to functions on them, this means that every ``compact quantum Lie group'' embeds in the compact quantum group associated to some $A_u(Q)$. In other words, collectively, the $A_u(Q)$'s play the same role in the world of compact quantum groups as the family of all unitary groups does in the theory of ordinary compact groups. 

It will be important to recall the structure of the fusion semiring $R_+$ of $A=A_u(Q)$, as worked out in \cite{MR1484551}. One of the main results of that paper is that there is a bijection $\bN*\bN\ni x\mapsto r_x\in R_+$ between the free monoid on two elements and the set of simple comodules, with multiplication in $R_+$ given by 
\begin{equation}\label{eq.fusion} 
r_xr_y = \sum_{x=ag,y=g^*b}r_{ab}. 
\end{equation} 
Here, `$*$' is the involutive anti-automorphim of $\bN*\bN$ interchanging the two copies of $\bN$, and the generators $\alpha$ and $\alpha^*$ of the two $\bN$'s correspond respectively to the fundamental comodule $V$ from the next-to-last paragraph and its dual.

%%%%%%%%%%%%%%%%%%%%%%%%%%%%%%%%%%%%%%%%%%%%%%%%%%%%%%%%%%%%%%%%%%%%%%%%%%%%%%%%%%%%%%%%%%%%%%%%%%%%%%%%%%%%%%%%%%

\subsection{Normal quantum subgroups}\label{subse.normal}

Always keeping in mind arrow reversal, a (closed) \define{quantum subgroup} of the (compact quantum group with) CQG algebra $A$ should be a quotient CQG algebra $A\to B$. This is indeed the standard definition in the literature, and the one we employ here. We regard the arrow itself as the quantum subgroup, and identify $A\to B$ and $A\to B'$ as quantum subgroups provided the two are isomorphic in the category of arrows in $\cqg$ sourced at $A$ (in the terminology of \cite[2.7]{MR2504527}, we identify quantum subgroups whenever they have the same imbedding).   

Following \cite[1.1.5]{MR1334152} and \cite[2.2]{MR2504527}, a quantum subgroup $\pi:A\to B$ of a compact quantum group is said to be \define{normal} if the corresponding right and left quantum coset spaces \[\lker(\pi)=\{ a\in A\ |\ \pi(a_1)\otimes a_2 = 1_B\otimes a \} \] and respectively \[\rker(\pi)=\{ a\in A\ |\ a_1\otimes \pi(a_2) = a\otimes 1_B \} \] coincide. The resulting linear subspace $\lker=\rker=C\le A$ then turns out to be a CQG subalgebra of $A$ and can be interpreted as functions on the quotient of the compact quantum group corresponding to $A$ by the normal quantum subgroup $\pi:A\to B$. 

It is perhaps worth pointing out that $\pi$ and $\iota$ determine each other: We have just seen how to get $\iota$ from $\pi$, and conversely, it can be shown (e.g. \cite[4.4]{MR2504527}) that $\pi:A\to B$ is precisely the quotient of $A$ by the ideal $AC^+=C^+A$, where $C^+=\ker(\varepsilon_C)$. What this means, in other words, is that the inclusion $iota:C\to A$ and the surjection $\pi:A\to B$ fit into an exact sequence $0\to C\to A\to B\to 0$ of Hopf (in our case also $*$-) algebras in the sense of \cite[1.2.3]{MR1334152}.

\begin{re}\label{rem.normal}
If $\iota:C\to P$ is to be half of an exact sequence, then $C$ must be invariant under the right (and also left) adjoint action of $P$ on itself given by $q\triangleleft p=S(p_1)qp_2$ (respectively $p\triangleright q=p_1qS(p_2)$). Indeed, it is observed in the proof of \cite[1.1.12]{MR1334152} that for any Hopf algebra morphism $f$, $\lker(f)$ as defined in \Cref{subse.normal} is invariant under the right adjoint action. It is for this reason that we refer to the CQG subalgebras $\iota:C\to P$ of interest, i.e. those giving rise to exact sequences, as \define{ad-invariant}.

In fact, invariance of a Hopf $*$-subalgebra $C\to P$ under either the left or right adjoint action is also sufficient in order that it be part of an exact sequence. This is proven in \cite[1.2.5]{MR1334152} provided $P$ is faithfuly flat over $C$; the latter condition always holds (that is, a CQG algebra is always faithfully flat over a CQG subalgebra) by \cite{2011arXiv1110.6701C}.   
\end{re}

In view of the above discussion, the following is very reasonable:

\begin{de}\label{def.simple}
A compact quantum group $A\in\cqg$ is \define{simple} if there are no normal quantum subgroups $\pi:A\to B$ apart from $\varepsilon_A$ and $\id_A$. 
\end{de}

\begin{ex}\label{ex.center}
As shown in \cite[4.5]{MR2504527}, sending $u_{ij}$ to $\delta_{ij}t$ implements a normal embedding of the one-dimensional torus $\bS^1$ with algebra of representative functions $C(\bS^1)=\mathbb C[t,t^{-1}]$ into any of the free unitary groups $A_u(Q)$. The aim of this paper is to prove that the resulting quotient is simple in the sense of \Cref{def.simple}. 

We denote the CQG algebra associated to this quotient by $P_u(Q)$, standing for `projective'. This is motivated by the fact that $P_u(Q)$ is a kind of ``projectivized'' version of the free unitary group. 
\end{ex}

\begin{re}\label{rem.fingen}
The terminology conflicts slightly with that of \cite[3.3]{MR2504527}: On the one hand, Wang's definition of simplicity only demands that there be no non-trivial \define{connected} normal quantum subgroups, and refers to the stronger form of simplicity from \Cref{def.simple} as \define{absolute}. I prefer the shorter term because there is no need for that distinction in this paper. In other ways though, \Cref{def.simple} might seem weaker than \cite[3.3]{MR2504527}, because it makes no mention of the other three conditions of the latter (numbered as in that paper): 

(1) In order to be simple, Wang requires that a CQG algebra be finitely generated. This is indeed not the case for $P_u(Q)$ (\Cref{rem.not_fingen}), but we remedy the problem in \Cref{prop.main}, where we provide a smaller, ``less canonical'' example of a simple CQG satisfying this additional condition.

(2) The simple CQG's of \cite{MR2504527} are required to be connected, in the sense that the $*$-subalgebra generated by any simple subcoalgebra is infinite-dimensional. $A_u(Q)$ and hence its CQG subalgebras are easily seen to satisfy this condition, so we need not worry about it any longer. 

(4) Simple CQG's are not supposed to have any non-trivial group-like elements, or equivalently, one-dimensional comodules, or again, one-dimensional subcoalgebras. Once more, this is automatically satisfied by $A_u(Q)$ and its CQG subalgebras, for example because the fusion ring is freely generated (as a ring) by the fundamental representation and its dual (\cite[Th\'eor\`eme 1 (ii)]{MR1484551}), which immediately shows that its only invertible elements are $\pm 1$.   
\end{re}

%%%%%%%%%%%%%%%%%%%%%%%%%%%%%%%%%%%%%%%%%%%%%%%%%%%%%%%%%%%%%%%%%%%%%%%%%%%%%%%%%%%%%%%%%%%%%%%%%%%%%%%%%%%%%%%%%%
%%%%%%%%%%%%%%%%%%%%%%%%%%%%%%%%%%%%%%%%%%%%%%%%%%%%%%%%%%%%%%%%%%%%%%%%%%%%%%%%%%%%%%%%%%%%%%%%%%%%%%%%%%%%%%%%%%

\section{Statements and proofs}\label{se.main}

Let us restate the result announced in the introduction:

\begin{theorem}\label{th.main}
For any positive invertible matrix $Q$, the compact quantum group $P_u(Q)$ is simple in the sense of \Cref{def.simple}.
\end{theorem}

The proof consists of showing that whenever a CQG subalgebra $\iota:C\to P=P_u(Q)$ fits into an exact sequence as explained in \Cref{subse.normal}, $C$ is either the scalars or the entire $P$. Since inclusions of CQG subalgebras induce inclusions of fusion semirings, we can attack the problem by limiting the possibilities for a fusion semiring of $R_+=R_+(P)$ if it is to correspond to such an embedding $\iota$. 

In working with fusion semirings, all of which are embedded in that of $A=A_u(Q)$ described in \Cref{subse.free}, we will identify the free monoid $\bN *\bN$ (and hence the set of simple comodules of $A$) with words in the generators $\alpha$ and $\alpha^*$. It will be less cumbersome, notationally, to substitute $0$ and $1$ respectively for $\alpha$ and $\alpha^*$, and write the elements of the free monoid as binary words; note that the `$*$' involution changes $0$'s into $1$'s and vice versa. First, we determine precisely which binary words correspond to simple comodules of $P_u(Q)$.

\begin{lemma}\label{lem.balanced}
The simple comodules of $P$ are parametrized by the binary words with equal numbers of $0$'s and $1$'s. 
\end{lemma}
\begin{proof}
Since by definition $P$ is the quotient of the free unitary group by the circle subgroup $A\to C(\bS^1)$, its simple comodules are precisely those which, when regarded as comodules over $C(\bS^1)$, break up as direct sums of the trivial comodule. The conclusion follows from the observations that (a) simple representations over the circle group are parametrized by $\bZ$, and (b) under this identification, scalar corestriction via $A\to C(\bS^1)$ turns the simple corresponding to a binary word $w$ into a direct sum of copies of the simple $C(\bS^1)$-comodule corresponding to the integer $(\sharp\text{ of }0\text{'s in }w)-(\sharp\text{ of }1\text{'s in }w)$.  
\end{proof}

We refer to binary words with equal numbers of $0$'s and $1$'s as \define{balanced}.

\begin{re}\label{rem.not_fingen}
Using the fusion rules \Cref{eq.fusion}, it can be shown that given any finite set $S$ of binary words, no simple in the semiring generated by $S$ can start with a longer contiguous segment of $0$'s than the longer such segment in a member of $S$. Together with \Cref{lem.balanced}, this shows that as noted in \Cref{rem.fingen}, $P$ is not finitely generated. 
\end{re}

We now start working towards proving \Cref{th.main}.

\begin{lemma}\label{lem.0110}
If $\iota:C\to P$ is ad-invariant and strictly larger than the scalars, then both $r_{01}$ and $r_{10}$ are in the fusion semiring $R_+(C)\subseteq R_+$ of $C$. 
\end{lemma}
\begin{proof}
By assumption, $C$ has some non-trivial simple comodule, whose corresponding binary word we may as well assume starts with a zero: $r_{0x}$. But then its dual will be $r_{x^*1}$, and it follows from the fusion rules \Cref{eq.fusion} that $r_{01}$ is a summand in the product $r_{0x}r_{x^*1}$. In conclusion, the simple $r_{01}$ must be in $R_+(C)$.

We now have to prove the same about $r_{10}$, and this is where ad-invariance comes into the picture. First, the argument in the previous paragraph solves the problem as soon as we can show that some $r_{1\cdots}$ is in $R_+(C)$ (then multiply it with its dual $r_{\cdots 0}$, etc.). To this end, fix some non-zero $p\in P$ in the coalgebra $C(r_{10})$, and let it act on a non-zero $c\in C$ in the coalgebra $C(r_{01})$ via the left adjoint action. On the one hand, the result must be in $C$ by ad-invariance. On the other, the fusion rules \Cref{eq.fusion} say that $r_{10}r_{01}r_{10}$ equals $r_{100110}$, and hence the multiplication map $C_{r_{10}}\otimes C_{r_{01}}\otimes C_{r_{10}}\to C_{r_{100110}}$ is an isomorphism. It follows from this that $p\triangleright c=p_1cS(p_2)$ is a non-zero element of $C_{r_{100110}}$ (note that $C_{r_{10}}$ is fixed by the antipode, so it contains both $p_1$ and $S(p_2)$). All in all, we get $r_{100110}\in R_+(C)$, and as observed at the beginning of this paragraph, this will do to finish the proof. 
\end{proof}

In the above proof and elsewhere we are tacitly using the correspondence between CQG subalgebras of a CQG algebra $A$ and sub-semirings of its fusion semiring $R_+(A)$. In one direction, an inclusion $\iota:C\to A$ of a CQG subalgebra induces a fully faithful monoidal functor between categories of comodules, and hence a fusion semiring inclusion. Moreover, $R_+(C)$ is the $\bN$-span of precisely those simples in $R_+(A)$ which correspond to simple subcoalgebras of $C$. In the other direction, given a sub-semiring $R_+$ of $R_+(A)$ which is an $\bN$-span of simples and is closed under the involution, the direct sum of simple subcoalgebras of $A$ corresponding to the simples in $R_+$ is a CQG subalgebra. These two constructions are inverse to one another, and implement the correspondence.

Let us stop for a moment to record the fact that we now already have an example of simple compact quantum group with non-commutative fusion semiring: The proof of \Cref{lem.0110} only uses $r_{01}$ and $r_{10}$ and the fusion rules of $A_u(Q)$, so it provides a proof for

\begin{proposition}\label{prop.main}
For any positive invertible matrix $Q$, the compact quantum group whose underlying CQG algebra is the subalgebra of $A_u(Q)$ generated by $u_{ij}u_{kl}^*$ and $u_{ij}^*u_{kl}$ is simple. \qedhere
\end{proposition}

As promised in \Cref{rem.fingen}, this example, unlike $P_u(Q)$, is finitely generated as an algebra and hence conforms to all of the conditions for simplicity in \cite[3.3]{MR2504527}. Apart from commutativity of the fusion semiring, this example felicitously lacks another property. It is observed in \cite[$\S$5]{MR2504527} that the examples of simple compact quantum groups mentioned in the introduction have what in that paper is called \define{property F}. It means that every embedding of a CQG subalgebra is part of an exact sequence, or, in view of \Cref{rem.normal}, every CQG subalgebra is invariant under the left adjoint action. The CQG algebra of \Cref{prop.main} clearly does not have property F, as its CQG subalgebra generated by $u_{ij}u_{kl}^*$, whose fusion semiring is generated by $r_{01}$, is not ad-invariant by the proof of \Cref{lem.0110}.

\begin{lemma}\label{lem.00111100}
Keeping the hypotheses and notation of \Cref{lem.0110}, all simples of the form $r_{0\cdots 01\cdots 1}$ and $r_{1\cdots 10\cdots 0}$ belong to $R_+(C)$. 
\end{lemma}
\begin{proof}
Let us prove the statement for $r_{0\cdots 01\cdots 1}$ ($n$ $0$'s and also $n$ $1$'s, since the binary word must be balanced). First, note that it is enough to show that some simple of the form $r_{0\cdots 0x}$ ($n$ $0$'s) is in $R_+(C)$. Indeed, its dual would correspond to the word $x^*$ followed by $n$ $1$'s, and their product would contain the desired simple as a summand. We know from \Cref{lem.0110} that $r_{10}$ is in $R_+(C)$, and now we can argue as in the proof of that lemma, acting on a non-zero element of $C_{r_{10}}$ by a non-zero element of $C_{r_{0\cdots 01\cdots 1}}$ via the left adjoint action to conclude.
\end{proof}

\begin{proof of main}
Keeping the notations and assumptions of \Cref{lem.0110}, we show by induction on the length of a balanced word $x\in\bN *\bN$ that $r_x$ belongs to $R_+(C)$. Having taken care of the base step of the induction in \Cref{lem.0110}, we can assume $x$ is at least four symbols long. There are two possibilities:

(1) $x$ consists of only two contiguous blocks of symbols, one of $0$'s and one of $1$'s, as in, say, $x=0\cdots 01\cdots 1$. This case is covered by \Cref{lem.00111100}.

(2) $x$ consists of more than two contiguous blocks of symbols. Write it without loss of generality as, say, $x=0\cdots 01y$, starting with $n\ge 1$ $0$'s. Then, by the fusion rules \Cref{eq.fusion}, $x$ is a summand in the product of $r_{0\cdots 01\cdots 1}$ ($n$ $0$'s) and $r_{0\cdots 0y}$ ($n-1$ $0$'s). These two words are both strictly shorter than $x$, so the induction hypothesis implies that both simples are in $R_+(C)$. This finishes the proof. 
\end{proof of main}

While these results partially address \cite[Problem 4.6 (2)]{MR2504527}, which asks for examples of simple compact quantum groups with non-commutative fusion semirings, it will still be interesting to investigate how much further \Cref{prop.main} can be pushed, and hence make some progress towards the classification of simple \define{quotients} of free unitary groups. Note that here `quotient quantum group' is used in the weak sense, meaning simply `CQG subalgebra'. Quotients in the stronger sense of left hand halves of exact sequences are taken care of by the following consequence of \Cref{th.main} (or rather of its proof):

\begin{corollary}\label{co.main}
Any proper normal quantum subgroup of $A=A_u(Q)$ is contained in the circle subgroup $A\to C(\bS^1)$ from \Cref{ex.center}.  
\end{corollary}
\begin{proof}
The previous results go through practically verbatim for $A$ (instead of $P=P_u(Q)$), and show that any non-trivial ad-invariant $\iota:C\to A$ contains the CQG subalgebra $P\subset A$. 
\end{proof}

We saw in \Cref{rem.normal} that CQG subalgebras $\iota:C\to H$ corresponding to quotients by normal quantum subgroups have a simple characterization as precisely those which are ad-invariant. It would be interesting though, as well as convenient, to have a purely combinatorial characterization in terms of fusion semirings:

\begin{question}
Let $\iota:C\to H$ be an inclusion of CQG algebras. Can the ad-invariance of $C$ be characterized solely in terms of the fusion ring inclusion $R_+(C)\subseteq R_+(H)$? 
\end{question}

Alternatively, and also probably more tractably,

\begin{question}
Does simplicity for a compact quantum group depend only on its fusion semiring? 
\end{question}

Positive answers would provide an alternative approach to the invariance of simplicity under deformation proved in \cite[$\S$5]{MR2504527}, and would be natural companions to such results as the possibility of lifting isomorphisms of fusion semirings to honest isomorphisms for compact connected Lie groups (\cite{MR1209960}) and the fusion semiring characterization of the center for a compact group (as in \cite{MR2130607} or \cite[3.9]{MR2383894}).

%%%%%%%%%%%%%%%%%%%%%%%%%%%%%%%%%%%%%%%%%%%%%%%%%%%%%%%%%%%%%%%%%%%%%%%%%%%%%%%%%%%%%%%%%%%%%%%%%%%%%%%%%%%%%%%%%%
%%%%%%%%%%%%%%%%%%%%%%%%%%%%%%%%%%%%%%%%%%%%%%%%%%%%%%%%%%%%%%%%%%%%%%%%%%%%%%%%%%%%%%%%%%%%%%%%%%%%%%%%%%%%%%%%%%

%\bibliography{simple}{}

\addcontentsline{toc}{section}{References}


\begin{thebibliography}{10}

\bibitem{MR1334152}
N.~Andruskiewitsch and J.~Devoto.
\newblock Extensions of {H}opf algebras.
\newblock {\em Algebra i Analiz}, 7(1):22--61, 1995.

\bibitem{MR1484551}
Teodor Banica.
\newblock Le groupe quantique compact libre {${\rm U}(n)$}.
\newblock {\em Comm. Math. Phys.}, 190(1):143--172, 1997.

\bibitem{MR1709109}
Teodor Banica.
\newblock Symmetries of a generic coaction.
\newblock {\em Math. Ann.}, 314(4):763--780, 1999.

\bibitem{MR2174219}
Teodor Banica.
\newblock Quantum automorphism groups of small metric spaces.
\newblock {\em Pacific J. Math.}, 219(1):27--51, 2005.

\bibitem{MR1937403}
Julien Bichon.
\newblock Quantum automorphism groups of finite graphs.
\newblock {\em Proc. Amer. Math. Soc.}, 131(3):665--673 (electronic), 2003.

\bibitem{2011arXiv1110.6701C}
A.~{Chirvasitu}.
\newblock {Cosemisimple Hopf algebras are faithfully flat over Hopf
  subalgebras}.
\newblock {\em ArXiv e-prints}, October 2011.

\bibitem{MR1786197}
Sorin D{\u{a}}sc{\u{a}}lescu, Constantin N{\u{a}}st{\u{a}}sescu, and
  {\c{S}}erban Raianu.
\newblock {\em Hopf algebras}, volume 235 of {\em Monographs and Textbooks in
  Pure and Applied Mathematics}.
\newblock Marcel Dekker Inc., New York, 2001.
\newblock An introduction.

\bibitem{MR1310296}
Mathijs~S. Dijkhuizen and Tom~H. Koornwinder.
\newblock C{QG} algebras: a direct algebraic approach to compact quantum
  groups.
\newblock {\em Lett. Math. Phys.}, 32(4):315--330, 1994.

\bibitem{MR934283}
V.~G. Drinfeld.
\newblock Quantum groups.
\newblock In {\em Proceedings of the {I}nternational {C}ongress of
  {M}athematicians, {V}ol. 1, 2 ({B}erkeley, {C}alif., 1986)}, pages 798--820,
  Providence, RI, 1987. Amer. Math. Soc.

\bibitem{MR1068703}
Michel Dubois-Violette and Guy Launer.
\newblock The quantum group of a nondegenerate bilinear form.
\newblock {\em Phys. Lett. B}, 245(2):175--177, 1990.

\bibitem{MR2383894}
Shlomo Gelaki and Dmitri Nikshych.
\newblock Nilpotent fusion categories.
\newblock {\em Adv. Math.}, 217(3):1053--1071, 2008.

\bibitem{MR1209960}
David Handelman.
\newblock Representation rings as invariants for compact groups and limit ratio
  theorems for them.
\newblock {\em Internat. J. Math.}, 4(1):59--88, 1993.

\bibitem{MR797001}
Michio Jimbo.
\newblock A {$q$}-difference analogue of {$U({\mathfrak g})$} and the
  {Y}ang-{B}axter equation.
\newblock {\em Lett. Math. Phys.}, 10(1):63--69, 1985.

\bibitem{MR1492989}
Anatoli Klimyk and Konrad Schm{\"u}dgen.
\newblock {\em Quantum groups and their representations}.
\newblock Texts and Monographs in Physics. Springer-Verlag, Berlin, 1997.

\bibitem{MR1116413}
Serge Levendorski{\u\i} and Yan Soibelman.
\newblock Algebras of functions on compact quantum groups, {S}chubert cells and
  quantum tori.
\newblock {\em Comm. Math. Phys.}, 139(1):141--170, 1991.

\bibitem{MR1243637}
Susan Montgomery.
\newblock {\em Hopf algebras and their actions on rings}, volume~82 of {\em
  CBMS Regional Conference Series in Mathematics}.
\newblock Published for the Conference Board of the Mathematical Sciences,
  Washington, DC, 1993.

\bibitem{MR2130607}
Michael M{\"u}ger.
\newblock On the center of a compact group.
\newblock {\em Int. Math. Res. Not.}, (51):2751--2756, 2004.

\bibitem{MR1216204}
Marc~A. Rieffel.
\newblock Compact quantum groups associated with toral subgroups.
\newblock In {\em Representation theory of groups and algebras}, volume 145 of
  {\em Contemp. Math.}, pages 465--491. Amer. Math. Soc., Providence, RI, 1993.

\bibitem{MR1382726}
Alfons Van~Daele and Shuzhou Wang.
\newblock Universal quantum groups.
\newblock {\em Internat. J. Math.}, 7(2):255--263, 1996.

\bibitem{MR1637425}
Shuzhou Wang.
\newblock Quantum symmetry groups of finite spaces.
\newblock {\em Comm. Math. Phys.}, 195(1):195--211, 1998.

\bibitem{MR2504527}
Shuzhou Wang.
\newblock Simple compact quantum groups. {I}.
\newblock {\em J. Funct. Anal.}, 256(10):3313--3341, 2009.

\bibitem{Woronowicz1987}
S.~L. Woronowicz.
\newblock Compact matrix pseudogroups.
\newblock {\em Comm. Math. Phys.}, 111(4):613--665, 1987.

\end{thebibliography}
%\bibliographystyle{plain}

\end{document}